\documentclass[11pt,twoside]{amsart}

\usepackage{amssymb}
\usepackage{latexsym}

\usepackage{color}

\catcode`\@=11
\def\omathop#1#2#3{\let\temp=#1\def\letter{#2}
  \ifcat#3_ \let\next\@@olim\else\let\next\@olim\fi\next#3}
\def\@olim{\letter\text{-}\!\temp}
\def\@@olim_#1{\mathchoice{
   \setbox0=\hbox{$\displaystyle\letter\text{-}\!\temp\!\text{-}\letter$}
   \setbox2=\hbox{$\displaystyle\temp$}
   \setbox4=\hbox{$\scriptstyle#1$}
   \dimen@=\wd4 \advance\dimen@ by -\wd2 \divide\dimen@ by2
   \def\next{\letter\text{-}\!\temp_{\hbox to 0pt{\hss$\scriptstyle#1$\hss}}
     \hskip\dimen@}
   \ifdim\wd2>\wd4 \def\next{\@olim_{#1}}\fi
   \ifdim\wd4>\wd0 \def\next{\mathop{\llap{$\letter$-}\!\temp}\limits_{#1}}\fi
   \next}
   {\@olim_{#1}}{\@olim_{#1}}{\@olim_{#1}}}

\def\bolim{\omathop{\lim}{bo}}

\catcode`\@=13

\newcommand{\reduce}{\mskip-2mu}
\newcommand{\ls}{\reduce\left\bracevert\reduce\vphantom{X}}
\newcommand{\rs}{\reduce\vphantom{X}\reduce\right\bracevert\reduce}

\theoremstyle{plain}
\newtheorem{thm}{Theorem}

\newtheorem{cor}[thm]{Corollary}
\newtheorem{lemma}[thm]{Lemma}
\newtheorem{rem}[thm]{Remark}

\theoremstyle{definition}
\newtheorem{definition}[thm]{Definition}

\begin{document}

\title[Dominated operators from a lattice-normed space  to a sequence vector lattice]{Dominated operators from a lattice-normed space  to a sequence vector lattice}

\author{Nariman~ Abasov}

\address{MATI -- Russian State Tecnological University\\
str. Orshanski  3,
Moscow, 121552 Russia}

\email{abasovn\copyright mail.ru}

\author{Abd El Monem Megahed}

\address{Department of Basic science, Faculty of Computers and Informatics\\
Suez Canal University\\
Ismailia, Egypt}

\email{$a_{-}$megahed15\copyright yahoo.com}

\author{Marat~Pliev}

\address{South Mathematical Institute of the Russian Academy of Sciences\\
str. Markusa 22,
Vladikavkaz, 362027 Russia}

\email{maratpliev\copyright gmail.com}

%\centerline{\today}



\begin{abstract}
We show that every dominated  linear operator from a Banach-Kantorovich space over atomless Dedekind complete vector lattice to a sequence Banach lattice $\ell_p(\Gamma)$ or $c_0(\Gamma)$ is narrow.   As a consequence, we obtain that an atomless Banach lattice cannot have a finite dimensional decomposition of a certain kind. Finally,  we  show  that if a  linear  dominated operator $T$ from lattice-normed space $V$ to  Banach-Kantorovich space $W$ is order narrow then the same  is its  exact dominant $\ls T\rs$.
\end{abstract}

\maketitle

\section{Introduction}

Narrow operators generalize compact operators defined on function spaces (see \cite{PP} for the first systematic study and a recent monograph \cite{PR}).
Different classes of narrow operators in framework of vector lattices and lattice-normed spaces   have considered by Popov and third named author in \cite{Pl-6,Pl-7,PP,PP-1}.
In the present paper we continue the of investigation of narrow operators in lattice-normed spaces and  show that,  every dominated linear operator from a Banach-Kantorovich space over an atomless Dedekind complete vector lattice to a sequence Banach lattice  is narrow.  As a consequence, we obtain that an atomless Banach lattice cannot have a finite dimensional decomposition of a certain kind. We  also prove   that if an  linear  dominated operator $T$ from lattice-normed space $V$ to  Banach-Kantorovich space $W$ is order narrow then the same  is its  exact dominant $\ls T\rs$. \footnote{The    third  named author was supported by the  Russian Foundation of Basic Research, the grant number 15-51-53119}.

\section{Preliminaries}

The  goal of this section is to introduce some basic definitions and facts. General information on vector lattices,
Banach spaces and lattice-normed spaces the reader can find in the books \cite{AK,AB,Ku,LT-1,LT-2}.

Consider a vector space $V$ and a real  archimedean vector lattice
$E$. A map $\ls \cdot\rs:V\rightarrow E$ is a \textit{vector norm} if it satisfies the following axioms:
\begin{enumerate}
  \item[1)] $\ls v \rs\geq 0;$\,\, $\ls v\rs=0\Leftrightarrow v=0$;\,\,$(\forall v\in V)$.
  \item[2)] $\ls v_1+v_2 \rs\leq \ls v_1\rs+\ls v_2 \rs;\,\, ( v_1,v_2\in V)$.
  \item[3)] $\ls\lambda  v\rs=|\lambda|\ls v\rs;\,\, (\lambda\in\Bbb{R},\,v\in V)$.
\end{enumerate}
A vector norm is called \textit{decomposable} if
\begin{enumerate}
  \item[4)] for all $e_{1},e_{2}\in E_{+}$ and $x\in V$ from $\ls x\rs=e_{1}+e_{2}$ it follows that there exist $x_{1},x_{2}\in V$ such that $x=x_{1}+x_{2}$ and $\ls x_{k}\rs=e_{k}$, $(k:=1,2)$.
\end{enumerate}

A triple $(V,\ls\cdot\rs,E)$ (in brief $(V,E),(V,\ls\cdot\rs)$ or $V$ with default parameters omitted) is a \textit{lattice-normed space} if $\ls\cdot\rs$ is a $E$-valued vector norm in the vector space $V$. If the norm $\ls\cdot\rs$ is decomposable then the space $V$ itself is called decomposable. We say that a net $(v_{\alpha})_{\alpha\in\Delta}$ {\it $(bo)$-converges} to an element $v\in V$ and write $v=\bolim v_{\alpha}$ if there exists a decreasing net $(e_{\gamma})_{\gamma\in\Gamma}$ in $E_{+}$ such that $\inf_{\gamma\in\Gamma}(e_{\gamma})=0$ and for every $\gamma\in\Gamma$ there is an index $\alpha(\gamma)\in\Delta$ such that $\ls v-v_{\alpha(\gamma)}\rs\leq e_{\gamma}$ for all $\alpha\geq\alpha(\gamma)$. A net $(v_{\alpha})_{\alpha\in\Delta}$ is called \textit{$(bo)$-fundamental} if the net $(v_{\alpha}-v_{\beta})_{(\alpha,\beta)\in\Delta\times\Delta}$ $(bo)$-converges to zero. A lattice-normed space is called {\it $(bo)$-complete} if every $(bo)$-fundamental net $(bo)$-converges to an element of this space. Let $e$ be a positive element of a vector lattice $E$.
By $[0,e]$ we denote the set $\{v\in V:\,\ls v\rs\leq e\}$.
A set $M\subset V$ is called  $\text{(bo)}$-{\it bounded } if there exists
$e\in E_{+}$ such that  $M\subset[0,e]$. Every decomposable $(bo)$-complete lattice-normed space is called a {\it Banach-Kantorovich space} (a BKS for short).

Let $(V,E)$ be a lattice-normed space.  A subspace $V_{0}$ of $V$ is called a $\text{(bo)}$-ideal of $V$ if for $v\in V$ and $u\in V_{0}$, from $\ls v\rs\leq\ls u\rs$ it follows that $v\in V_{0}$. A subspace $V_{0}$ of a decomposable lattice-normed space $V$  is a $\text{(bo)}$-ideal if and only if $V_{0}=\{v\in V:\,\ls v\rs\in L\}$, where $L$ is an order ideal in $E$ (see \cite{Ku}, Proposition~2.1.6.1). Let $V$ be a lattice-normed space and $y,x\in V$. If $\ls x\rs\bot\ls y\rs=0$ then we call the elements $x,y$ {\it disjoint} and write $x\bot y$. The equality $x=\coprod\limits_{i=1}^{n}x_{i}$ means that $x=\sum\limits_{i=1}^{n}x_{i}$ and $x_{i}\bot x_{j}$ if $i\neq j$. An element $z\in V$ is called a {\it component} or a \textit{fragment} of $x\in V$ if $0\leq \ls z\rs\leq\ls x\rs$ and $x\bot(x-z)$. Two fragments $x_{1},x_{2}$ of $x$  are called \textit{mutually complemented} or $MC$, in short, if $x=x_1+x_{2}$. The notations $z\sqsubseteq x$ means that $z$ is a fragment of $x$. The set of all fragments of an element $v\in V$ is denoted by $\mathcal{F}_{v}$. According to (\cite{AB}, p.111) an element $e>0$ of a vector lattice $E$ is called an {\it atom}, whenever $0\leq f_{1}\leq e$, $0\leq f_{2}\leq e$ and $f_{1}\bot f_{2}$ imply that either $f_{1}=0$ or $f_{2}=0$. A vector lattice $E$ is  atomless if there is no atom $e\in E$.

Consider some important examples of lattice-normed spaces. We begin with simple extreme cases, namely vector lattices and normed spaces. If $V=E$ then the modules of an element can be taken as its lattice norm: $\ls v\rs:=|v|=v\vee(-v);\,v\in E$. Decomposability of this norm easily follows from the Riesz Decomposition Property holding in every vector lattice. If $E=\Bbb{R}$ then $V$ is a normed space.

Let $Q$ be a compact  and let $X$ be a Banach space. Let $V:=C(Q,X)$ be the space of  continuous vector-valued functions from $Q$ to $X$. Assign $E:=C(Q,\Bbb{R})$. Given $f\in V$, we define its lattice norm by the relation $\ls f\rs:t\mapsto\|f(t)\|_{X}\,(t\in Q)$. Then $\ls\cdot\rs$ is a decomposable norm (\cite{Ku}, Lemma~2.3.2).

Let $(\Omega,\Sigma,\mu)$ be a $\sigma$-finite measure space,
let $E$ be an order-dense ideal in  $L_{0}(\Omega)$ and let $X$ be a Banach space.
By $L_{0}(\Omega,X)$ we denote the space of (equivalence classes of)  Bochner $\mu$-measurable vector functions
acting from $\Omega$ to $X$. As usual, vector-functions are equivalent if they have
equal values at almost all points of the set $\Omega$. If $\widetilde{f}$ is the coset of a measurable vector-function $f:\Omega\rightarrow X$ then $t\mapsto\|f(t)\|$,$(t\in\Omega)$ is a scalar measurable function whose coset is denoted by the symbol $\ls\widetilde{f}\rs\in L_{0}(\mu)$. Assign by definition
$$
E(X):=\{f\in L_{0}(\mu,X):\,\ls f\rs\in E\}.
$$
Then $(E(X),E)$ is a lattice-normed space with a decomposable norm (\cite{Ku}, Lemma~2.3.7.).
If $E$ is a Banach lattice then the lattice-normed space $E(X)$ is  a Banach space with respect to the norm $|\|f|\|:=\|\|f(\cdot)\|_{X}\|_{E}$.

Let $E$ be a Banach lattice and let $(V,E)$ be a lattice-normed space. By definition,
$\ls x\rs\in E_{+}$ for every $x\in V$, and we can introduce some \textit{mixed norm} in $V$ by the formula
$$
\||x|\|:=\|\ls x\rs\|\,\,\,(\forall\, x\in V).
$$
The normed space $(V,\||\cdot|\|)$ is called a \textit{space with a mixed norm}.
In view of the inequality $|\ls x\rs-\ls y\rs|\leq\ls x-y\rs$ and monotonicity of the norm in $E$, we have
$$
\|\ls x\rs-\ls y\rs\|\leq\||x-y|\|\,\,\,(\forall\, x,y\in V),
$$
so a vector norm is a norm continuous operator from $(V,\||\cdot|\|)$ to $E$. A lattice-normed space $(V,E)$ is called
a \textit{Banach space with a mixed norm} if the normed space $(V,\||\cdot|\|)$ is complete with respect to the norm convergence.

Consider lattice-normed spaces $(V,E)$ and $(W,F)$, a linear operator $T:V\rightarrow W$ and a positive operator
$S\in L_{+}(E,\,F)$. If the condition
$$
\ls Tv\rs\leq
S\ls v\rs;\,(\forall\, v\in V)
$$
is satisfied then we say that $S$ \textit{dominates} or
\textit{majorizes} $T$ or that $S$ is \textit{dominant}
or {majorant} for $T$.
In this case $T$ is called a \textit{dominated} or
\textit{majorizable} operator.  The set of all dominants of the operator $T$
is denoted by $\text{maj}(T)$. If there is the least element in
$\text{maj}(T)$ with respect to the order induced by $L_{+}(E,F)$ then
it is called the {\it least} or the {\it exact dominant} of $T$ and it is denoted by
$\ls T\rs$. The set of all dominated operators from $V$ to $W$ is denoted by $M(V,W)$.

Narrow operators in lattice-normed spaces were firstly introduced  in \cite{Pl-6}. Let us give some definitions.
\begin{definition} \label{def:nar1}
Let $(V,E)$ be a  lattice-normed space over an  atomless vector lattice and $X$ be a Banach space. A linear operator $T: V \to X$ is called:
\begin{itemize}
\item \emph{order-to-norm} continuous,   if $T$ sends $(bo)$-convergent nets in $V$ to  norm convergent nets in $X$;
\item \emph{narrow }  if for every $v\in V$ and $\varepsilon > 0$ there exist $MC$ fragments $v_1$, $v_2$ of $v$ such that $\|Tv_1-Tv_2\|<\varepsilon$;
\item \emph{strictly narrow } if for every $v\in V$ there exist $MC$ fragments $v_1$, $v_2$ of $v$ such that $Tv_1 = Tv_2  $.
\end{itemize}
\end{definition}

Next is the definition of an order narrow operator.

\begin{definition} \label{def:nar2}
Let $(V,E)$ and $(W,F)$ be  lattice-normed spaces  with $E$ atomless. A linear operator  $T:V\to W$ is called
\begin{itemize}
  \item \emph{order narrow }  if for every $v\in V$ there exists a net of decompositions $v = v_\alpha^{1} \sqcup v_\alpha^{2}$ such that $(Tv_{\alpha}^{1} - Tv_{\alpha}^{2} )\overset{(bo)}\longrightarrow 0$.
\end{itemize}
\end{definition}

\section{Dominated operators to sequences spaces }

In this section we investigate narrow operators from Banach-Kantorovich space to sequences Banach lattices.

Our first result here is the following theorem.

\begin{thm} \label{thm:newww}
Let $(V,E)$ be a Banach-Kantorovich space over an atomless Dedekind complete vector lattice $E$ and $\Gamma$ any set. Let $X = X(\Gamma)$ denote one of the Banach lattices $c_0(\Gamma)$ or $\ell_p(\Gamma)$ with $1 \leq p < \infty$. Then every order-to-norm continuous linear dominated  operator $T: V \to X$ is narrow.
\end{thm}

For the proof, we need the following lemma (\cite{Pl-6}, Lemma~4.11).

\begin{lemma} \label{le:w88495737}
Let $(V,E)$ be a Banach-Kantorovich space over an atomless Dedekind complete vector lattice $E$, $F$ a finite dimensional Banach space. Then  every order-to-norm continuous linear  operator $T: V \to F$ is narrow.
\end{lemma}

\begin{proof}[Proof of Theorem~\ref{thm:newww}]
Let  $T: V \to X$ be a dominated operator, $v \in V$ and $\varepsilon > 0$. Take an arbitrary $u\in\mathcal{F}_v$. Since  $u\bot(v-u)$ we have the  estimation:
\begin{align*}
\ls Tu\rs\leq \ls Tu\rs+\ls T(v-u)\rs\leq\\
\leq\ls T\rs\ls u\rs+\ls T\rs\ls v-u\rs=\ls T\rs\ls v\rs=f\in F_{+}.
\end{align*}
Then we choose a finite subset  $\Gamma_0 \subset \Gamma$, such that
\begin{enumerate}
  \item $|f(\gamma)| \leq \varepsilon/4$ для любого  $\gamma \in \Gamma \setminus \Gamma_0$, if $X = c_0(\Gamma)$ и
  \item $\sum\limits_{\gamma \in \Gamma \setminus \Gamma_0} (f(\gamma) )^p \leq (\varepsilon/4)^p$ if $X = \ell_p(\Gamma)$.
\end{enumerate}

Let $P$ be the projection of $X$ onto $X(\Gamma_0)$ along $X(\Gamma \setminus \Gamma_0)$ and $Q = Id - P$ the orthogonal projection. Obviously, both $P$ and $Q$ are positive linear bounded operators. Since $S = P \circ T: V \to X(\Gamma_0)$ is a finite rank order-to-norm continuous dominated operator, by Lemma~\ref{le:w88495737}, $S$ is narrow, and hence, there are MC fragments $v_1$, $v_2$ of $v$ with $\|S(v_1) - S(v_2)\| < \varepsilon/2$.  Since  $|T(v_i)| \leq f$, by the positivity of   $Q$ we have that  $Q(Tv_i) \leq Qf$ and   $\|Q(T(v_i))\| \leq \|Q(f)\|$ for   $i = 1,2$. Moreover, by (1) and (2) $\|Q(f)\| \leq \varepsilon/4$. Then
\begin{align*}
\|T(v_1) - T(v_2)\| &= \|S(v_1) + Q(T(v_1)) - S(v_2) - Q(T(v_2))\| \\
&\leq \|S(v_1) - S(v_2)\| + \|Q(T(v_1))\| + \|Q(T(v_2))\| \\
&< \, \frac{\varepsilon}{2} \, + \|Q(f)\| + \|Q(f)\| < \varepsilon.
\end{align*}
\end{proof}

For a space with mixed norm   we obtain the following consequence of Theorem~\ref{thm:newww}.

\begin{cor} \label{cor:BL}
Let $(V,E)$ be  a Banach space with mixed norm over  an atomless order complete and order continuous Banach lattice $E$ and $\Gamma$ any set. Let $X = X(\Gamma)$ denote one of the Banach lattices $c_0(\Gamma)$ or $\ell_p(\Gamma)$ with $1 \leq p < \infty$. Then every continuous dominated linear  operator $T: V \to X$ is narrow.
\end{cor}
\begin{proof}
It is enough to prove that every continuous operator $T:V\to X$ is order-to-norm continuous. Take a net $(v_{\alpha})_{\alpha\in\Lambda}$ which $(bo)$-convergent   to zero. Since $E$ is order continuous Banach lattice, we have $\||v_{\alpha}|\|=\|\ls v_{\alpha}\rs\|\longrightarrow 0$.
\end{proof}
The idea used in the proof of Theorem~\ref{thm:newww} could be generalized as follows.

\begin{definition}
Let $E$, $F$ be ordered vector spaces. We say that a linear operator $T: E \to F$ is \textit{quasi-monotone with a constant} $M > 0$ if for each $x,y \in E^+$ the inequality $x \leq y$ implies $Tx \leq M Ty$. An operator $T: E \to F$ is said to be \textit{quasi-monotone} if it is quasi-monotone with some constant $M > 0$.
\end{definition}

If $T \neq 0$ in the above definition, we easily obtain $M \geq 1$. Observe also that the quasi-monotone operators with constant $M = 1$ exactly are the positive operators.

Recall that a sequence of elements $(e_n)_{n=1}^\infty$ (resp., of finite dimensional subspaces $(E_n)_{n=1}^\infty$) of a Banach space $E$ is called a \textit{basis} (resp., a \textit{finite dimensional decomposition}, or \textit{FDD}, in short) if for every $e \in E$ there exists a unique sequence of scalars $(a_n)_{n=1}^\infty$ (resp. sequence $(u_n)_{n=1}^\infty$ of elements $u_n \in E_n$) such that $e = \sum\limits_{n=1}^\infty a_n e_n$ (resp., $e = \sum\limits_{n=1}^\infty u_n$). Every basis $(e_n)$ generates the FDD $E_n = \{\lambda e_n: \, \lambda \in \mathbb R\}$. Any basis $(e_n)$ (resp., any FDD $(E_n)$) of a Banach space generates the corresponding \textit{basis projections} $(P_n)$ defined by
$$
P_n \Bigl( \sum\limits_{k=1}^\infty a_k e_k \Bigr) = \sum\limits_{k=1}^n a_k e_k \,\,\,\,\, \left( \mbox{resp.,} \,\, P_n \Bigl( \sum\limits_{k=1}^\infty u_k \Bigr) = \sum\limits_{k=1}^n u_k \right),
$$
which are uniformly bounded. For more details about these notions we refer the reader to \cite{LT-1}. The orthogonal projections to $P_n$'s defined by $Q_n = Id - P_n$, where $Id$ is the identity operator on $E$, we will call the \textit{residual projections} associated with the basis $(e_n)_{n=1}^\infty$ (resp., to the FDD $(E_n)_{n=1}^\infty$).

\begin{definition}
A basis $(e_n)$ (resp., an FDD $(E_n)$) of a Banach lattice $E$ is called \textit{residually quasi-monotone} if there is a constant $M > 0$ such that the corresponding residual projections are quasi-monotone with constant $M$.
\end{definition}

In other words, an FDD $(E_n)$ of $E$ is residually quasi-monotone if the corresponding approximation of smaller in modulus elements is better, up to some constant multiple: if $x,y \in E$ with $|x| \leq |y|$ then $\|x - P_nx\| \leq M \|y - P_ny\|$ for all $n$ (observe that $\|z - P_nz\| \to 0$ as $n \to \infty$ for all $z \in E$).

\begin{thm} \label{thm:newww2}
Let $(V,E)$ be a Banach-Kantorovich space over an atomless Dedekind complete vector lattice $E$ and $F$ be a Banach lattice with a residually quasi-monotone basis or, more general, a residually quasi-monotone FDD. Then every dominated linear  operator $T: V \to F$ is narrow.
\end{thm}

\begin{proof}
Let $(F_n)$ be an FDD of $F$ with the corresponding projections $(P_n)$, and let $M > 0$ be such that for every $n \in \mathbb N$ the operator $Q_n = Id - P_n$ is quasi-monotone with constant $M$. Let $T: V \to F$ be a dominated linear  operator, $v \in V$ and $\varepsilon > 0$. Choose $f \in F_+$ so that $|Tx| \leq f$ for all $x \sqsubseteq v$. Since $\lim\limits_{n \to \infty} P_n f = f$, we have that $\lim\limits_{n \to \infty} Q_n f = 0$. Choose $n$ so that
\begin{equation} \label{eq:djjf7}
\|Q_n f\| \leq \frac{\varepsilon}{4M} \, .
\end{equation}
Since $S = P_n \circ T: V \to E_n$ is a finite rank dominated linear operator, by Lemma~\ref{le:w88495737}, $S$ is narrow, and hence, there are MC fragments $v_1$, $v_2$ of $v$ such that $\|Sv_1 - Sv_2\| < \varepsilon/2$. Since $|Tv_i| \leq f$, by the quasi-monotonicity of $Q_n$ we have that $\|Q_n(Tv_i)\| \leq M \|Q_n f\|$ for $i = 1,2$. Then by \eqref{eq:djjf7},
\begin{align*}
\|Tv_1 - Tv_2\| &= \|Sv_1 + Q(Tv_1) - Sv_2 - Q(Tv_2)\| \\
&\leq \|Sv_1 - Sv_2\| + \|Q ( T v_1)\| + \|Q ( T v_2)\| \\
&< \, \frac{\varepsilon}{2} \, + M\|Qf\| + M\|Qf\| < \varepsilon.
\end{align*}
\end{proof}

\begin{cor}
An atomless order continuous Banach lattice $E$ cannot admit a residually quasi-monotone FDD.
\end{cor}
\begin{proof}
Banach lattice $E$ is a lattice-normed space $(E,E)$, where vector norm coincide with module.
Thus, it is enough to observe that the identity operator of such a Banach lattice is not narrow.
\end{proof}

\section{Domination problem for narrow operators}
\label{sec5}

In this section we consider a domination problem for the exact dominant  of dominated linear  operators. The domination problem can be stated as follows. Let $(V,E)$, $(W,F)$ be lattice-normed spaces, $T: V \to W$ be linear operator and $\ls T\rs:E\to F$ be its exact dominant. Let $\mathcal P$ be some property of linear operators, so that $\mathcal P(R)$ means that $R$ possesses $\mathcal P$. Does $\mathcal P(T)$ imply $\mathcal P(\ls T\rs)$ and vice versa? Some recent results in this direction the reader can find in \cite{Get,Pl-7}.

Recall that vector lattice $E$ is said to possesses the {\it strong Freudenthal property}, if for  $f,e\in E$ such that $|f|\leq \lambda|e|$, $\lambda\in \Bbb{R}_{+}$ $f$ can be $e$-uniformly approximated by linear
combinations $\sum\limits_{k=1}^{n}\lambda_{k}\pi_{k}e$, where $\pi_{1},\dots,\pi_{n}$ are order projections in $E$.

Denote by $E_{0+}$ the conic hull of the set $\ls V\rs=\{\ls v\rs:\,v\in V\}$, i.e., the set of elements of the form $\sum\limits_{k=1}^{n}\ls v_{k}\rs$, where $v_{1},\dots,v_{n}\in V$, $n\in\Bbb{N}$.

\begin{thm}[\cite{Ku}, Theorem~4.1.18.] \label{dom:01}
Let $(V,E),(W,F)$ be lattice-normed spaces. Suppose $E$  possesses the  strong Freudenthal property, $V$ is decomposable and $F$ is order complete. Then every
dominated operator has the exact dominant $\ls T\rs$.
The exact dominant of an arbitrary operator $T\in M(V,W)$ can be calculated by the following formulas:
\begin{gather}
\ls T\rs(e)=\sup\left\{\sum\limits_{i=1}^n\ls Tv_i\rs:
\sum\limits_{i=1}^n\ls v_i\rs= e, \ls v_{i}\rs\bot\ls v_{j}\rs;\,i\neq j;\,n\in\Bbb{N}; \,e\in E_{0+}\right\};\\
\ls T\rs(e)=\sup\{\ls T\rs(e_{0}):\,e_{0}\in E_{0+};\, e_{0}\leq e\}
(e\in E_{+});\\
\ls T\rs(e)=\ls T\rs(e_+)+\ls T\rs(e_-),\,
(e\in E).
\end{gather}
\end{thm}

\begin{thm}
Let $(V,E)$, $(W,F)$ be  Banach-Kantorovich spaces,  $E,F$ be Dedekind  complete, $E$ be an  atomless and with the  strong Freudenthal property and   $T$ be a dominated linear  operator from $V$ to $W$. If $T$ is   order narrow,  then the same  is its exact dominant $\ls T\rs:E\to F$.
\end{thm}

\begin{proof}
At first fix any $e\in E_{0+}$ and $\varepsilon>0$.  Since
$$
\left\{\sum\limits_{i=1}^n\ls Tv_i\rs:
\sum\limits_{i=1}^n\ls v_i\rs= e;\,\ls v_{i}\rs\bot\ls v_{j}\rs;\,i\neq j;\,n\in\Bbb{N}\right\}
$$
is an increasing net,  there exits  a net of finite collections $\{v^{\alpha}_{1},\dots,v^{\alpha}_{n_{\alpha}}\}\subset V$, $\alpha\in\Lambda$ with
$$
e=\bigsqcup\limits_{i=1}^{n_{\alpha}}\ls v_{i}^{\alpha}\rs,\,\alpha\in\Lambda
$$
and
$$
\Big(\ls T\rs(e)-\sum\limits_{i=1}^{n_{\alpha}}\ls Tv_{i}^{\alpha}\rs\,\Big)\leq y_{\alpha}\overset{(o)}\longrightarrow 0,
$$
where $0\leq y_{\alpha}$, $\alpha\in\Lambda$ is an decreasing net and $\inf(y_{\alpha})_{\alpha\in\Lambda}=0$. Fix some $\alpha\in\Lambda$. Since $T$ is order narrow operator we may assume that there exist a finite set of a nets of a decompositions $v_{i}^{\alpha}=u_{i}^{\beta_{\alpha}}\sqcup w_{i}^{\beta_{\alpha}}$, $i\in\{1,\dots,n_{\alpha}\}$ which depends of $\alpha$, indexed by the same set $\Delta$, such that $\ls Tu_{i}^{\beta_{\alpha}}-Tw_{i}^{\beta_{\alpha}}\rs\overset{(bo)}\longrightarrow 0$, $i\in\{1,\dots, n_{\alpha}\}$. Let $f^{\beta_{\alpha}}=\coprod\limits_{i=1}^{n_{\alpha}}\ls u_{i}^{\beta_{\alpha}}\rs$ and $g^{\beta_{\alpha}}=\coprod\limits_{i=1}^{n_{\alpha}}\ls w_{i}^{\beta_{\alpha}}\rs$. Then we have
\begin{align*}
0\leq\ls T\rs(f)-\sum\limits_{i=1}^{n_{\alpha}}(\ls Tu_{i}^{\beta_{\alpha}}\rs)\leq\\
\ls T\rs(e)-\sum\limits_{i=1}^{n_{\alpha}}\ls Tv_{i}^{\alpha}\rs;\\
0\leq\ls T\rs(g)-\sum\limits_{i=1}^{n_{\alpha}}(\ls Tw_{i}^{\beta_{\alpha}}\rs)\leq\\
\ls T\rs(e)-\sum\limits_{i=1}^{n_{\alpha}}\ls Tv_{i}^{\alpha}\rs.
\end{align*}
Now we may write
\begin{align*}
\Big|\ls T\rs f^{\beta_{\alpha}}-\ls T\rs g^{\beta_{\alpha}}\Big|=\\
\Big|\ls T\rs f^{\beta_{\alpha}}-\sum\limits_{i=1}^{n_{\alpha}}\ls Tu_{i}^{\beta_{\alpha}}\rs+
\sum\limits_{i=1}^{n_{\alpha}}\ls Tu_{i}^{\beta_{\alpha}}\rs-
\sum\limits_{i=1}^{n_{\alpha}}\ls Tw_{i}^{\beta_{\alpha}}\rs+
\sum\limits_{i=1}^{n_{\alpha}}\ls Tw_{i}^{\beta_{\alpha}}\rs-\ls T\rs g^{\beta_{\alpha}}\Big|\leq\\
\Big|\ls T\rs f^{\beta_{\beta_{\alpha}}}-\sum\limits_{i=1}^{n_{\beta_{\alpha}}}\ls Tu_{i}^{\beta_{\alpha}}\rs\Big|+
\Big|\ls T\rs g^{\beta_{\alpha}}-\sum\limits_{i=1}^{n_{\alpha}}\ls Tw_{i}^{\alpha}\rs\Big|+
\Big|\sum\limits_{i=1}^{n_{\alpha}}\ls Tu_{i}^{\beta_{\alpha}}\rs-\sum\limits_{i=1}^{n_{\alpha}}\ls Tw_{i}^{\beta_{\alpha}}\rs\Big|\leq\\
2\Big(\ls T\rs(e)-\sum\limits_{i=1}^{n_{\alpha}}\ls Tv_{i}^{\alpha}\rs\Big)+
\sum\limits_{i=1}^{n_{\alpha}}\Big|\ls Tu_{i}^{\beta_{\alpha}}\rs-\ls Tw_{i}^{\beta_{\alpha}}\rs\Big|\leq\\
2\Big(\ls T\rs(e)-\sum\limits_{i=1}^{n_{\alpha}}\ls Tv_{i}^{\alpha}\rs\Big)+
\sum\limits_{i=1}^{n_{\alpha}}\ls Tu_{i}^{\beta_{\alpha}}- Tw_{i}^{\beta_{\alpha}}\rs\overset{(o)}\longrightarrow 0.
\end{align*}
Therefore $e=f^{\beta_{\alpha}}\sqcup g^{\beta_{\alpha}}$, $\alpha\in\Lambda$, $\beta_{\alpha}\in\Delta$ is a desirable net of decompositions. Now, let $e\in E_{+}$. Observe that
$D=\{f\leq e:\,f\in E_{0+}\}$
is a directed set.
Indeed, let $f_{1}=\coprod\limits_{i=1}^{k}\ls u_{i}\rs, f_{1}\leq e$,  $f_{2}=\coprod\limits_{j=1}^{n}\ls w_{j}\rs; f_{2}\leq e$, $u_{i}, w_{j}\in V$, $1\leq i\leq k$, $1\leq j\leq n$. Then by decomposability of the vector norm there exists a set  of mutually disjoint elements $(v_{ij})$, $1\leq i\leq k$, $1\leq j\leq n$, such that
$u_{i}=\coprod\limits_{j=1}^{n}v_{ij}$ for every $1\leq i\leq k$ and $w_{j}=\coprod\limits_{i=1}^{k}v_{ij}$ for every $1\leq j\leq n$. Let $f=\coprod\ls v_{ij}\rs$. It is clear that $\ls T\rs f_{i}\leq\ls T\rs f$, $i\in\{1,2\}$. Let $(e_{\alpha})_{\alpha\in\Lambda}, e_{\alpha}\in D$ be a net, where $\ls T\rs=\sup\limits_{\alpha}\ls T\rs e_{\alpha}$. Fix $\alpha\in\Lambda$, then for $e_{\alpha}\in D$ there exists a net of decompositions $e_{\alpha}=f_{\alpha}^{\beta}\sqcup g_{\alpha}^{\beta}$, $\beta\in\Delta$, such that $\Big|\ls T\rs f_{\alpha}^{\beta}-\ls T\rs g_{\alpha}^{\beta}\Big|\overset{(o)}\longrightarrow 0$. Thus we have
\begin{align*}
\Big| \ls T\rs(e-e_{\alpha}+f_{\alpha}^{\beta})-\ls T\rs g_{\alpha}^{\beta}\Big|=\\
\Big| \ls T\rs(e-e_{\alpha})+\ls T\rs f_{\alpha}^{\beta}-\ls T\rs g_{\alpha}^{\beta}\Big|\leq\\
 \Big(\ls T\rs e-\ls T\rs e_{\alpha}+
 \Big|\ls T\rs f_{\alpha}^{\beta}-\ls T\rs g_{\alpha}^{\beta}\Big|\Big)\overset{(o)}\longrightarrow 0.
\end{align*}
Hence $e=((e-e_{\alpha})\sqcup f_{\alpha}^{\beta}))\sqcup g_{\alpha}^{\beta}$ is a desirable net of decompositions.
Finally for arbitrary element $e\in E$ we have $e=e_{+}-e_{-}$ and by the Theorem~\ref{dom:01},(4) we have $\ls T\rs(e)=\ls T\rs(e_{+})-\ls T\rs(e_{-})$. Thus, if $e_{+}=f_{1}^{\alpha}\sqcup f_{2}^{\alpha}$ and $e_{-}=g_{1}^{\alpha}\sqcup g_{2}^{\alpha}$ are necessary nets of decompositions, then
\begin{align*}
\Big|\ls T\rs(f_{1}^{\alpha}+g_{1}^{\alpha})-\ls T\rs(f_{2}^{\alpha}+g_{2}^{\alpha})\Big|=\\
\Big|\ls T\rs f_{1}^{\alpha}-\ls T\rs f_{2}^{\alpha}+\ls T\rs g_{1}^{\alpha}-\ls T\rs g_{2}^{\alpha})\Big|\leq\\
\Big(\Big|\ls T\rs (f_{1}^{\alpha}- f_{2}^{\alpha})\Big|+\Big|\ls T\rs( g_{1}^{\alpha}-g_{2}^{\alpha})\Big|\Big)\overset{(o)}\longrightarrow 0
\end{align*}
and $e=(f_{1}^{\alpha}+g_{1}^{\alpha})\sqcup(f_{2}^{\alpha}+g_{2}^{\alpha})$ is a desirable net of decompositions.
\end{proof}
\begin{rem}
This  is an open question. Does the order narrowness of the operator $\ls T\rs$   implies the order narrowness of the  $T$?
A particular case was proved  in (\cite{Pl-6}, Theorem~5.1).
\end{rem}

\end{document}